\documentclass[12pt,a4paper]{article}
\usepackage{amssymb,amsthm,amsmath,color,fullpage,url,hyperref}
\usepackage{graphicx}

\newtheorem{thm}{Theorem}
\newtheorem{lem}[thm]{Lemma}
\newtheorem{obs}[thm]{Observation}

\newtheorem{cor}[thm]{Corollary}

\newcommand{\mc}[1]{\mathcal{#1}}
\newcommand{\GG}{\mc{G}}
\newcommand{\BB}{\mc{B}}
\newcommand{\bb}[1]{\mathbb{#1}}
\newcommand{\brm}[1]{\operatorname{#1}}

\newcommand{\diam}{\operatorname{diam}}
\newcommand{\mesh}{\operatorname{mesh}}
\newcommand{\asdim}{\operatorname{asdim}}
\begin{document}
\title{Asymptotic dimension of intersection graphs}

\author{Zden\v{e}k Dvo\v{r}\'ak
	\thanks{Charles University, Prague, Czech Republic. E-mail: \protect\href{mailto:rakdver@iuuk.mff.cuni.cz}{\protect\nolinkurl{rakdver@iuuk.mff.cuni.cz}}.
	Supported by the ERC-CZ project LL2005 (Algorithms and complexity within and beyond bounded expansion) of the Ministry of Education of Czech Republic.}
		\and Sergey Norin \thanks{McGill University, Montr\'{e}al, Quebec, Canada.  E-mail: \protect\href{mailto: 
			sergey.norin@mcgill.ca}{\protect\nolinkurl{sergey.norin@mcgill.ca}}. Supported by an NSERC Discovery grant.}
		}

\date{}

\maketitle

\begin{abstract}
We show that intersection graphs of compact convex sets in $\mathbb{R}^n$ of bounded aspect ratio
have asymptotic dimension at most $2n+1$.  More generally, we show this is the case for
intersection graphs of systems of subsets of any metric space of Assouad-Nagata dimension $n$
that satisfy the following condition: For each $r,s>0$ and every point $p$, the number of pairwise-disjoint
elements of diameter at least $s$ in the system that are at distance at most $r$ from $p$ is bounded by a function of $r/s$.
\end{abstract}

The \emph{Lebesgue covering dimension} of a topological space is the smallest integer $n$ such that
the space has arbitrarily fine open covers in which each point belongs to at most $n+1$ covering sets.
The covering dimension is a local property, describing the small-scale behavior of the space,
and as such is not an interesting parameter of discrete spaces such as graphs.  Gromov~\cite{gromov1993asymptotic}
introduced a dual notion of \emph{asymptotic dimension}, describing the large-scale behavior.
Several recent results indicate the asymptotic dimension is interesting for graph classes:
\begin{itemize}
\item All graph classes of bounded treewidth have asymptotic dimension at most one,
and all proper minor-closed classes have asymptotic dimension at most two~\cite{bonamy2021asymptotic}.
\item Graph classes with polynomial growth of degree $n$ have asymptotic dimension $O(n)$,
as (implicitly) shown by Krauthgamer and Lee~\cite{KraLee07}.
\item On the other hand, bounded-degree expanders have unbounded asymptotic dimension (implicitly 
shown in~\cite{gromov2003random}).
\end{itemize}
We continue the investigation of the asymptotic dimension of graph classes, focusing on geometric graphs.  Intuitively, one would expect the dimension of geometric graphs to be closely related to the dimension of the space in which they
are represented.  We study a particular type of geometric graphs, the \emph{intersection graphs}
whose vertices are objects in a Euclidean space and two objects are adjacent if and only if they intersect.
Of course, any graph is an intersection graph of some set system, and thus in order to obtain any meaningful result,
we need to restrict the families of objects that we consider.

On the negative side, some seemingly strong restrictions on the geometric representation are not sufficient.  For example, consider the
class $\BB'$ of graphs obtained from bipartite graphs by subdividing each edge once.
Every graph in $\BB'$ can be represented as an intersection graph of axis-aligned boxes in $\mathbb{R}^3$, with the vertices
of the two parts of the original bipartite graph represented by long perpendicular boxes that nearly touch, and the subdivision
vertices by small cubes between them.  And still, the class $\BB'$ has unbounded asymptotic dimension,
since it contains arbitrarily large subcubic expanders (indeed, $\BB'$ contains all graphs obtained from cubic expanders by subdividing each edge
three times, and such graphs are also expanders).

However, we show that the dimension is bounded as long as we forbid such ``long but narrow'' objects,
and only consider the intersection graphs of ``fat'' objects such as balls or cubes.  As a specific
example, we obtain the following bound on the asymptotic dimension of the class of intersection graphs of convex sets with bounded aspect ratio. (The \emph{height} of a subset $S$ of $\mathbb{R}^n$
is the infimum of $h\ge 0$ such that $S$ is contained between two parallel hyperplanes at distance $h$,
and the \emph{aspect ratio} of $S$ is the ratio of the diameter of $S$ to its height).

\begin{thm}\label{thm-mainsimp}
For every positive integer $n$ and every $\alpha\ge 1$, the class of intersection graphs of systems of compact convex sets of aspect ratio at most $\alpha$
in $\mathbb{R}^n$ has asymptotic dimension at most $2n+1$.
\end{thm}

We prove this result in greater generality:
\begin{itemize}
\item Instead of $\mathbb{R}^n$, we consider objects in any metric space of Assouad-Nagata dimension~$n$.
\item Instead of the ``convex set with bounded aspect ratio'' condition, we use a more abstract condition
of being ``space-filling'', roughly saying that it cannot be the case that many large disjoint sets from the system
are close to any single point.
\end{itemize}
Note that this is not just a generalization for the sake of generality; it turns out that operating
in these terms makes our arguments much cleaner than our initial formulations specific to Euclidean spaces.

Let us now give definitions necessary to state the result precisely.
Let $(X,d)$ be a metric space. Let $\mc{U}$ be a collection of subsets of $X$.
The \emph{Lebesgue number} $L(\mc{U})$ is the supremum of $r \geq 0$ such that every
ball $B(x,r)$ of radius $r$ is contained in some $U \in \mc{U}$. We denote by
$\mesh(\mc{U})$ the supremum of diameters of elements of $\mc{U}$.

A function $D:\mathbb{R}_+\to\mathbb{R}_+$ is an \emph{$n$-dimensional control function}
for a metric space $(X,d)$ if for all $r>0$, there is a cover $\mc{U} = \bigcup_{i=1}^{n+1}\mc{U}_i$ of $X$
such that $\mesh(\mc{U}) \leq D(r)$, $L(\mc{U}) \geq r$ and for each $i$,
the elements of $\mc{U}_i$ are pairwise-disjoint.
The \emph{asymptotic dimension} $\asdim(X)$ of a space $(X,d)$ is the smallest integer $n$ such that
there exists an $n$-dimensional control function for $(X,d)$.  We will also need a related notion of
\emph{Assouad-Nagata dimension} $\asdim_{AN}(X)$, which is the minimum $n$ for which
there exists $K>1$ such that $D(r)=Kr$ is an $n$-dimensional control function for $(X,d)$.

Note that for asymptotic dimension, it suffices to verify the condition
for all sufficiently large $r$ (say $r\ge 1$), as we can set $D(r)=D(1)$ for $r<1$.
On the other hand, for the Assouad-Nagata dimension, the form of $D$ is fixed for all $r>0$
(and since $\lim_{r\to 0} D(r)=0$, the Lebesgue covering dimension of the space is at most
as large as its Assouad-Nagata dimension, see~\cite[Proposition 2.2]{lang2005nagata} for details
of the argument).  Let us also remark that the Euclidean space $\mathbb{R}^n$ has Assouad-Nagata dimension $n$.

We study the asymptotic dimension of graph metrics; we view a graph $G$ as a metric
space $(V(G),d_G)$, where $d_G(u,v)$ is the minimum length of a path between $u$ and $v$ in~$G$.
Each finite graph has asymptotic dimension $0$,
as each of its components can be covered by a single set of bounded radius; hence, rather than single
graphs, we consider the asymptotic dimension of graph classes.  The \emph{asymptotic dimension} of a graph class $\GG$
is defined as the asymptotic dimension of the disjoint union of all graphs in $\GG$; equivalently, $\GG$
has asymptotic dimension at most $n$ if some function $D:\mathbb{R}_+\to\mathbb{R}_+$
is an $n$-dimensional control function for all graphs in $\GG$.

More specifically, we consider the intersection graphs of objects in spaces of bounded Assouad-Nagata dimension.
For a system $\mc{S}$ of sets, the \emph{intersection graph} of $\mc{S}$ is the graph with vertex set $\mc{S}$
and with two vertices $S_1,S_2\in \mc{S}$ adjacent if and only if $S_1\cap S_2\neq\emptyset$.  
As noted above, to obtain meaningful results on asymptotic dimension of classes of intersection graphs we need to restrict 
the set systems we consider.
Motivated by the intersection graphs of
balls in Euclidean spaces, we define the following notion.  For a function $f:\mathbb{R}_+\to\mathbb{Z}_+$,
the system $\mc{S}$ of subsets of a metric space $(X,d)$ is \emph{$f$-space-filling} if all sets in $\mc{S}$ have finite diameter and for every $r,s>0$ and $x\in X$,
the ball $B(x,r)$ of radius $r$ around $x$ is intersected by at most $f(r/s)$ pairwise-disjoint elements of $\mc{S}$
of diameter at least $s$.

We can now state our main  result, which generalizes  Theorem~\ref{thm-mainsimp}.

\begin{thm}\label{thm-main}
Let $(X,d)$ be a metric space of Assouad-Nagata dimension $n$. For every $f:\mathbb{R}_+\to\mathbb{Z}_+$,
the family of intersection graphs of $f$-space-filling systems of subsets of $X$ has asymptotic dimension at most $2n+1$.
\end{thm}

Let us show that Theorem~\ref{thm-main} indeed implies Theorem~\ref{thm-mainsimp}.
A natural way to ensure the space-filling property is by requiring the sets to cover a substantial portion of the space around
each of their points.  We say that a set $S$ is \emph{$\eta$-round} if $S$ has finite diameter and for every $v \in S$ and every $r \leq \diam(S)$,
there exists $v' \in S$ such that $$B(v',\eta r) \subseteq  S \cap B(v,r).$$
We say that a metric space $(X,d)$ is a \emph{doubling space} if there exists an integer $K$ such that
for every $r>0$, every ball of radius $r$ in $(X,d)$ can be covered by at most $K$ balls of radius $r/2$.
The minimum such $K$ is the \emph{doubling constant} of the space.  The simplest examples
of doubling spaces are the Euclidean spaces $\mathbb{R}^n$, with doubling constant $\exp(O(n))$.
However, there exist doubling spaces that do not embed in Euclidean spaces with bounded distortion~\cite{nondoub}.

\begin{obs}
For $\eta>0$ and a positive integer $K$, let us define $f(x)=K^{\lceil \log_2 \frac{x+1}{\eta}\rceil}$.  Then any system of $\eta$-round
sets in a metric space $(X,d)$ with doubling constant $K$ is $f$-space-filling.
\end{obs}
\begin{proof}
Consider any $r,s>0$ and $x\in X$, and let $S_1$, \ldots, $S_m$ be pairwise-disjoint $\eta$-round subsets of $X$
of diameter at least $s$ intersecting $B(x,r)$.  For $i=1,\ldots, m$, let $v_i$ be a point in $S_i\cap B(x,r)$,
By the definition of $\eta$-roundness, there exists $v'_i\in S_i$ such that
$B(v'_i,\eta s) \subseteq  S_i \cap B(v_i,s)\subseteq S_i\cap B(x,r+s)$.
Since $S_1$, \ldots, $S_m$ are pairwise-disjoint, so are the balls $B(v'_1,\eta s)$, \ldots, $B(v'_m,\eta s)$,
and thus for $i\neq j$, the distance between $v'_i$ and $v'_j$ is greater than $2\eta s$.  In particular,
any ball of radius $\eta s$ contains at most one of the points $v'_1$, \ldots, $v'_m$.  Since $(X,d)$ has doubling constant $K$,
the ball $B(x,r+s)$ is covered by at most $K^{\lceil \log_2 \frac{r+s}{\eta s}\rceil}=f(r/s)$ balls of radius $\eta s$,
and thus $m\le f(r/s)$.
\end{proof}

Moreover, doubling spaces have finite Assouad-Nagata dimension.
\begin{lem}[{Lang and Schlichenmaier~\cite[Lemma 2.3]{lang2005nagata}}]\label{lem-dan}
A metric space with doubling constant $K$ has Assouad-Nagata dimension at most $K^3$.
\end{lem}

Hence, Theorem~\ref{thm-main} has the following consequence, which implies Theorem~\ref{thm-mainsimp} since
each compact convex set in $\mathbb{R}^n$ of aspect ratio at most $\alpha$ is
$\tfrac{1}{2\alpha n}$-round\footnote{By~\cite[Lemma 9]{subconvex}, each compact convex subset $S$ of $\mathbb{R}^n$ of height $h$
contains a ball $A$ of radius $r_0=\tfrac{h}{2n}$.  Consider any point $v\in S$ and let $b\le \diam(S)$ be the distance between $v$ and the furthest point
in $A$.  Consider any $r\leq \diam(S)$, and let $A'=v+\tfrac{\min(b,r)}{b}(A-v)$.  By convexity, the ball $A'$ is contained in $S$,
and the choice of $b$ and $A'$ implies $A'\subseteq B(v,r)$.  Moreover, $A'$ has radius at least
$$\frac{\min(b,r)}{b}\cdot r_0=\frac{\min(b,r)}{br}\cdot r_0\cdot r=\frac{h}{\max(b,r)}\cdot\frac{1}{2n}\cdot r\ge \frac{1}{2\alpha n}\cdot r.$$}.

\begin{cor}\label{cor-main}
For every $\eta > 0$, the family $\GG$ of intersection graphs of systems of $\eta$-round sets in a 
doubling space $(X,d)$
has finite asymptotic dimension.  If additionally $(X,d)$ has Assouad-Nagata dimension $n$, then $\GG$ has asymptotic dimension at most $2n+1$.
\end{cor}

Before we proceed with the proof of Theorem~\ref{thm-main}, which occupies the rest of the paper, let us give some remarks:
\begin{itemize}
\item We cannot replace asymptotic dimension by Assouad-Nagata dimension in the outcome, even in Corollary~\ref{cor-main}.
For example, the class of intersection graphs of balls in $\mathbb{R}^3$ has unbounded Assouad-Nagata dimension.
Indeed, for a graph $G$ and a positive integer $k$ let $G^{[k]}$ denote the graph obtained from $G$ by subdividing
each edge exactly $k$ times, let $\GG$ be a class of graphs, let $n:\GG\to\mathbb{Z}^+$ be any function and let $\GG'=\{G^{[n(G)]}:G\in\GG\}$.
Then $\asdim_{AN}(\GG)\le \asdim_{AN}(\GG')$, see~\cite[Lemma 6.1]{bonamy2021asymptotic}.
In particular, if $\GG$ is the class of all graphs, then $\GG'$ has unbounded Assouad-Nagata dimension.  And finally, it is easy to see
if $n(G)$ is large enough, then $G^{[n(G)]}$ can be represented as an intersection graph of balls in $\mathbb{R}^3$.
\item Conversely, we cannot replace Assouad-Nagata dimension by asymptotic dimension in the assumptions  since asymptotic dimension
does not control behavior on small scales; it is possible to have a space with bounded asymptotic dimension but unbounded
Lebesgue covering dimension, and the intersection graph could be realized in these small-scale high-dimensional parts.
This can be worked around by assuming that both the asymptotic dimension and the Lebesgue covering dimension are bounded,
and by modifying the definition of the space-filling property to reflect the non-linear control function.  However, the arguments
then become substantially more technical and we do not have any application for this more general setting, and thus we choose
to present the results in terms of Assouad-Nagata dimension.
\item In Corollary~\ref{cor-main}, it is not enough to assume roundness of the sets and finite Assouad-Nagata dimension,
i.e., we cannot drop the assumption that the doubling constant is bounded.  Indeed, consider an infinite graph $G$
of unbounded asymptotic dimension, and let $G'$ be the metric space obtained from $G$ by adding a universal vertex and
then interpreting all edges as intervals of length $1$.  Then $G'$ has Assouad-Nagata dimension $1$.
Let $\mc{S}$ be the system of balls of radius $1/2$ around the points of $G'$ corresponding to the vertices of $G$.
The elements of this system are $\tfrac{1}{2}$-round, but the intersection graph of $\mc{S}$ is isomorphic to $G$
and consequently has unbounded asymptotic dimension.  A standard compactness argument shows that
the class of finite intersection graphs of systems of $\tfrac{1}{2}$-round sets in $G'$ has unbounded asymptotic dimension.
\item The bound $2n+1$ likely is not optimal.  For example, the interval graphs (the intersection graphs of
balls in $\mathbb{R}^1$) have asymptotic dimension $1$ rather than $3$.  It is natural to ask whether it
is possible to decrease the bound to $n$ in Theorem~\ref{thm-main} (or at least in Theorem~\ref{thm-mainsimp}).
Let us remark that $n$-dimensional grids have asymptotic dimension $n$ and can be represented as touching graphs of balls in $\mathbb{R}^n$,
and thus it is not possible to decrease the bound below $n$.
\end{itemize}

The proof of Theorem~\ref{thm-main} is divided among Sections~\ref{sec-space}--\ref{sec-wrapup}, with the bulk of the work accomplished in Sections~\ref{sec-webe} and~\ref{sec-treedec}. In Section~\ref{sec-webe} we prove that  any set system satisfying the conditions of Theorem~\ref{thm-main} can be partitioned into $n$ parts with every part admitting a ``well-behaved'' tree decomposition. In Section~\ref{sec-treedec} we extend the argument from~\cite{bonamy2021asymptotic} to show that the class of graphs admitting such  tree  decompositions have asymptotic dimension one.   

\section{Space-filling}\label{sec-space}

In this section, we show that the space-filling property is preserved by bounded-radius unions.

A set $S$ is a \emph{$b$-shallow union} of elements of $\mc{S}$ if there exists $\mc{S}'\subseteq \mc{S}$
such that $S$ is the union of the sets in $\mc{S}'$ and the intersection graph of $\mc{S}'$ has radius at most $b$.

\begin{lem}\label{lem-cons}
For a function $f:\mathbb{R}_+\to\mathbb{Z}_+$, let $\mc{S}$ be an $f$-space-filling system of subsets of a metric space $(X,d)$.
Let $b$ be a positive integer and let $h(x)=f((2b+2)(x+1))$.
If $\mc{S}'$ is a system of $b$-shallow unions of elements of $\mc{S}$, then $\mc{S}'$ is $h$-space-filling.
\end{lem} 
\begin{proof}
Consider any $r,s>0$ and $x\in X$, and let $S_1$, \ldots, $S_m$ be pairwise-disjoint elements of $\mc{S}'$ of diameter at least $s$
intersecting $B(x,r)$. We need to show that $m \leq h(r/s)$.

For $i=1,\ldots, m$, since $S_i$ is a $b$-shallow union of elements of $\mc{S}$, there exists
$\mc{S}_i\subseteq \mc{S}$ such that $S_i=\bigcup \mc{S}_i$ and the intersection graph $H_i$ of $\mc{S}_i$ has radius at most $b$.
Note that $s\le \diam(S_i)\le (2b+1)\sup\{\diam(U):U\in \mc{S}_i\}$, and thus there exists $U_i\in \mc{S}_i$ of diameter
at least $s/(2b+2)$.  Moreover, there exists $U'_i\in \mc{S}_i$ containing a point $p_i\in B(x,r)$.
Choose $U_i$ and $U'_i$ so that their distance in $H_i$ is the smallest possible.  Note that all vertices on the
shortest path from $U'_i$ to $U_i$ except for $U_i$ have diameter less than $\frac{s}{2b+2}$ and the path has length at most $2b$, and thus
$d(p_i,U_i)\le \frac{2bs}{2b+2}\le s$.  Therefore, $U_i$ intersects $B\bigl(x,r+s\bigr)$.

Since $S_1$, \ldots, $S_m$ are pairwise-disjoint, $U_1$, \ldots, $U_m$ are also pairwise-disjoint, and since $\mc{S}$ is $f$-space-filling,
we conclude that
$$m\le f\left(\frac{r+s}{s/(2b+2)}\right)=h(r/s),$$
as desired.
\end{proof}

\section{Webs and tree decompositions}\label{sec-webe}

Let $(X,d)$ be a metric space.  For a real number $C>1$, we say that a set $W\subseteq X$ with finite diameter \emph{$C$-catches} a non-empty set $S\subseteq X$ 
if $S \subseteq W$ and $\brm{diam}(W) \leq C \diam (S)$.  A system $\mc{W}$ of subsets of $X$ \emph{$C$-catches} the set $S$
if some element of $\mc{W}$ $C$-catches $S$.
A \emph{$C$-web} over $(X,d)$ is a system of non-empty subsets of $X$ that $C$-catches every 
non-empty subset $S$ of $X$ with finite diameter.
A system of non-empty sets $\mc{U}$ is \emph{laminar} if for all $S,T \in \mc{U}$ we have $S \cap T \in \{\emptyset,S,T\}$.
A system is \emph{$n$-laminar} if it is a union of at most $n$ laminar systems. 

\begin{lem}\label{lemma-exweb}
	Let $(X,d)$ be a metric space of Assouad-Nagata dimension at most $n$, with $n$-dimensional control function $D(r)=Kr$ for some $K>1$.
	Then $(X,d)$ admits an $(n+1)$-laminar $C$-web, where $C=2K(2K+1)$.
\end{lem}	
\begin{proof} 
	By the definition of the Assouad-Nagata dimension, for every integer $\ell$ there exists a cover
	$\mc{U^{\ell}} = \cup_{i=1}^{n+1}\mc{U}^\ell_i$ of $X$ such that $\mesh(\mc{U}^\ell) \leq K(2K+1)^{\ell}$,
	$L(\mc{U}^\ell) \geq (2K+1)^{\ell}$ and for each $i$, the elements of $\mc{U}^\ell_i$ are pairwise-disjoint.

        Consider now a fixed $i\in\{1,\ldots,n+1\}$.
	For each $U\in \mc{U}^\ell_i$, let $R(U)$ be the set of $V \in \cup_{\ell' < \ell}\:\mc{U}^{\ell'}_i$ for which there exist
	sequences $U=Z_0,Z_1,\ldots,Z_m=V$ and $\ell=\ell_0>\ell_1>\ldots>\ell_m$ such that $m\ge 1$ and for $j=1,\ldots,m$,
we have	$Z_j\in \mc{U}^{\ell_j}_i$, $Z_j\cap Z_{j-1}\neq\emptyset$, and $Z_j\not\subseteq Z_{j-1}$.
	For each set $V\in R(U)$, fix such sequences $(Z_0,\ldots,Z_m)$ and $(\ell_0,\ldots,\ell_m)$ arbitrarily and choose a point $p_{U,V}\in Z_1\setminus U$.
	Since $\diam(Z_j)\le K(2K+1)^{\ell_j}\le K(2K+1)^{\ell-j}$, the distance from any point of $V$ to $p_{U,V}$ is at most
	$$\sum_{j\ge 1} K(2K+1)^{\ell-j}=\frac{1}{2}(2K+1)^{\ell}.$$

	Let $$\tilde{U} = U \setminus \bigcup_{V\in R(U)} V$$
	and let $\tilde{\mc{U}}^{\ell}_i = \{\tilde{U}: U \in \mc{U}^{\ell}_i\}$.  Let $\tilde{\mc{U}}^{-\infty}_i=\{\{x\}:x\in X\}$,
and $\mc{W}_i = \cup_{\ell\in\mathbb{Z} \cup \{-\infty\}} \tilde{\mc{U}}^{\ell}_i$.

	We claim that the system $\mc{W}_i$ is laminar.
	Indeed, consider distinct sets $\tilde{U}_1,\tilde{U}_2\in \mc{W}_i$, where $\tilde{U}_1\in \tilde{\mc{U}}^a_i$, $\tilde{U}_2\in \tilde{\mc{U}}^b_i$,
	and $\tilde{U}_1\cap \tilde{U}_2\neq \emptyset$.  Note that $a\neq b$, as otherwise $\tilde{U}_1\subseteq U_1$ and $\tilde{U}_2\subseteq U_2$
	would be disjoint; hence, without loss of generality we have $a<b$. Note that $U_1 \subseteq U_2$, as otherwise, $U_1 \in R(U_2)$ and so $\tilde{U}_2 \subseteq U_2 \setminus U_1$, in contradiction to the  assumption  $\tilde{U}_1\cap \tilde{U}_2\neq \emptyset$. 
	
	It suffices to show that $\tilde{U}_1 \subseteq \tilde{U}_2$. 
	Suppose not, then there exists  $V\in R(U_2)$ intersecting $\tilde{U}_1$. Let $U_2=Z_0,Z_1,\ldots, Z_m=V$ be the sequence showing that $V\in R(U_2)$.
	Let $j$ be the largest index such that $Z_j\not\subseteq U_1$; note that $j\ge 1$, since $U_1\subseteq U_2$ and $Z_1\not\subseteq U_2$.
	We have $Z_j\cap U_1\neq \emptyset$, since $V\cap U_1\neq\emptyset$ and if $j<m$, then $Z_{j+1}\subseteq U_1$ and $Z_j\cap Z_{j+1}\subseteq Z_j\cap U_1$ is non-empty.
	Moreover, $U_1\not\subseteq Z_j$, since $Z_j\in R(U_2)$ and $\tilde{U}_1\cap \tilde{U}_2\subseteq U_1\cap (U_2\setminus Z_j)$ is non-empty.
	Therefore, $U_1\in R(Z_j)$ or $Z_j\in R(U_1)$.  The former is not possible, since $U_1\not\in R(U_2)\supseteq R(Z_j)$.  Hence, $Z_j\in R(U_1)$, and thus $V\in R(Z_j)\subseteq R(U_1)$. It follows that $ \tilde{U}_1 \cap V = \emptyset$,  a contradiction finishing the proof of our claim.

        Therefore, the system $\mc{W}=\bigcup_{i=1}^{n+1} \mc{W}_i$ is $(n+1)$-laminar.
	It remains to show that $\mc{W}$ is a $C$-web. That is we need to show that for every non-empty subset $S$ of $X$ with finite diameter there exists $\tilde{U}\in\mc{W}$ such that $\tilde{U}$ $C$-catches $S$.  

	If $\diam(S)=0$, then $S$ consists of a single point and $S\in\mc{W}$.
	Let $r=\diam(S)>0$, then there exists unique $\ell\in Z$ such that $\frac{1}{2} (2K+1)^{\ell-1}\le r < \frac{1}{2} (2K+1)^{\ell}$.
	Choose a point $x\in S$ arbitrarily; we have $S\subseteq B(x,r)$.
	Since $L(\mc{U}^\ell) \geq (2K+1)^{\ell}$, the ball $B(x,r+\frac{1}{2}(2K+1)^\ell)$ is contained in some
	$U \in \mc{U}^{\ell}$.  Recall that for every $V \in R(U)$, every point of $V$ is at distance
	at most $\frac{1}{2}(2K+1)^{\ell}$ from a point $p_{U,V}$ not belonging to $U$, and thus $V$ is disjoint from $B(x,r)$.
	Therefore, $S\subseteq B(x, r)\subseteq \tilde{U}$.  Moreover,
	$$\diam(\tilde{U})\le\diam(U)\le \mesh(\mc{U}^\ell) \leq K(2K+1)^{\ell}\le 2K(2K+1)r.$$
	Therefore $\tilde{U}\in\mc{W}$ $C$-catches $S$, as desired.
\end{proof}

Laminar systems naturally give tree decompositions
for intersection graphs of set systems sets caught by them.  Moreover, if such a set system satisfies the conditions of Theorem \ref{thm-main}, then its bags have asymptotic dimension zero. In the remainder of this section, we make these observations precise.

We start by recalling the necessary definitions.
A \emph{tree decomposition} of a graph $H$ is a pair $(T,\beta)$, where $T$ is a tree and $\beta$ is a function
assigning a \emph{bag} $\beta(x)\subseteq V(H)$ to each node $x\in V(T)$, such that
\begin{itemize}
\item for each $uv\in E(H)$, there exists $x\in V(T)$ such that $\{u,v\}\subseteq \beta(x)$, and
\item for each $v\in V(H)$, $\{x\in V(T):v\in\beta(x)\}$ induces a connected non-empty subtree of $T$.
\end{itemize}
A set $D$ of vertices of a graph $G$ is \emph{dominating} if every vertex in $V(G)\setminus D$ has a neighbor in $D$.
For a positive integer $k$, a tree decomposition $(T,\beta)$ of a graph $H$ is \emph{$k$-dominated} if
$H[\beta(x)]$  has a dominating set of size at most $k$ for every $x \in V(T)$.

\begin{lem}\label{lemma-decomp}
Let $\mc{W}$ be a laminar system of subsets of a 
metric space $(X,d)$.  For a function
$f:\mathbb{R}_+\to\mathbb{Z}_+$, let $\mc{S}$ be a finite $f$-space-filling system of subsets of $X$.
Let $C>1$ be a real number.  If every set in $\mc{S}$ is $C$-caught by $\mc{W}$, then
the intersection graph of $\mc{S}$ has an $f(C)$-dominated tree decomposition.
\end{lem}
\begin{proof}
For every $S\in \mc{S}$, choose arbitrarily a set $W'_S\in \mc{W}$ that $C$-catches $S$.
Let $T$ be the rooted tree with vertex set $\{X\}\cup \{W'_S:S\in \mc{S}\}$, the root $X$,
and with $W_1\in V(T)$ being the parent of $W_2\in V(T)$ if and only if $W_2\subsetneq W_1$
and no element $W\in V(T)\setminus \{W_1,W_2\}$ satisfies $W_2\subsetneq W\subsetneq W_1$
(laminarity of $\mc{W}$ implies that the graph $T$ defined in this way is indeed a tree).

For every $S\in \mc{S}$, let $W_S$ be the vertex of $T$ such that $S\subseteq W_S$ and the
distance between the root and $W_S$ is maximum; note that this vertex is unique by laminarity,
and that $W_S$ is a descendant of $W'_S$ in $T$.  For every $W\in V(T)$, let us define
$\beta(W)$ as the set of vertices $S\in \mc{S}$ such that  $W\cap S\neq\emptyset$ and $W \subseteq W_S$, that is a $W$ descendant  of $W_S$ in $T$ (possibly $W=W_S$).

We claim that $(T,\beta)$ is a tree decomposition of the intersection graph $H$ of $\mc{S}$.
Note that for each $S\in \mc{S}$, we have $S\in\beta(W_S)$, and that the nodes of $T$ whose bags
contain $S$ form a subtree of $T$ rooted in $W_S$.  Moreover, consider any edge $S_1S_2\in E(H)$.
Since $S_1\cap S_2\neq\emptyset$, $S_1\subseteq W_{S_1}$ and $S_2\subseteq W_{S_2}$,
we have $W_{S_1}\cap W_{S_2}\neq\emptyset$.  Since $\mc{W}$ is laminar, we can assume
$W_{S_1}\subseteq W_{S_2}$, and thus $W_{S_1}$ is a descendant of $W_{S_2}$ in $T$.
Moreover, $S_1\cap S_2\subseteq W_{S_1}\cap S_2$ is non-empty, and thus
$S_2\in \beta(W_{S_1})$.  Therefore, $\{S_1,S_2\}\subseteq\beta(W_{S_1})$, and thus
$(T,\beta)$ is indeed a tree decomposition of $H$.

Let us now consider any node $W$ of $T$.  If $S\in \beta(W)$, then $W\subseteq W_S\subseteq W'_S$,
and since $W'_S$ $C$-catches $S$, we have $\diam(W)\le \diam(W'_S)\le C\diam(S)$.
Let $D$ be a maximal independent set in $H[\beta(W)]$; then the elements of $D$ are pairwise-disjoint and $D$ is a dominating set in $H[\beta(W)]$.
Let $x$ be any point of $W$.  Let $r=\diam(W)$ and note that each $S\in D$ intersects $W\subseteq B(x,r)$
and $\diam(S)\ge r/C$. 
Thus $|D|\le f(C)$
and  the tree decomposition $(T,\beta)$ is $f(C)$-dominated.
\end{proof}

\section{Weak diameter coloring and dominated tree decompositions}\label{sec-treedec}

The asymptotic dimension can be described in terms of \emph{weak diameter coloring}, as follows.
Given a (not necessarily proper) coloring $\varphi$ of a graph $G$, a subgraph $H$ of $G$ is \emph{monochromatic}
if all its vertices receive the same color in $\varphi$.  The \emph{$G$-diameter} of $H$ is the maximum distance in $G$ between
vertices of $H$.  The \emph{$G$-diameter of a coloring} of a subgraph $F$ of $G$ is defined as the maximum $G$-diameter of a monochromatic
connected subgraph of $F$.  A \emph{$c$-coloring} is a coloring using only colors $\{1,\ldots,c\}$.
For a graph $G$ and a positive integer $r$, let $G^r$ denote the graph with vertex set $V(G)$ in which two
vertices are adjacent if and only if the distance between them in $G$ is at most $r$.

\begin{lem}[{\cite[Proposition 1.17]{bonamy2021asymptotic}}]\label{lemma-asweak}
The class $\GG$ of graphs has asymptotic dimension at most $n$ if and only
if there exists a function $D':\mathbb{Z}_+\to\mathbb{Z}_+$ such that for every positive integer $r$
and for every $G\in\GG$, the graph $G^r$ has an $(n+1)$-coloring of $G^r$-diameter at most $D'(r)$.
\end{lem}

Combining Lemma~\ref{lemma-asweak} with Lemma~\ref{lem-cons}, allows us to reduce  the proof of Theorem~\ref{thm-main} to
establishing existence of bounded $G$-diameter $(2n+2)$-colorings of the intersection graphs  under consideration
(indeed, note that by Lemma~\ref{lem-cons}, if $G$ is an intersection graph of a space-filling system of subsets
and $r$ is odd, then $G^r$ is also an intersection graph of a space-filling system, and that validity
of the conclusion in Lemma~\ref{lemma-asweak} for $r$ even follows from the validity for $r+1$).
In this section, we describe how to obtain a $2$-coloring of bounded $G$-diameter for any graph
with dominated tree decomposition. This is the  final ingredient of the proof of Theorem~\ref{thm-main}.

\begin{lem}\label{lem-col}
There exists a function $w:\mathbb{Z}_+ \to \mathbb{Z}_+$ such that for every integer $k\ge 0$,
every graph $G$ with a $k$-dominated tree decomposition has a $2$-coloring of $G$-diameter at most $w(k)$.
\end{lem}

The proof of Lemma~\ref{lem-col} is a variation of the proof of~\cite[Theorem 1.4]{bonamy2021asymptotic}.
For the purposes of induction, we actually prove a strengthening of Lemma~\ref{lem-col}, stated as Lemma~\ref{lem-col1} below.

We need some definitions and auxiliary results.  We say that a set of vertices $Z$ is \emph{$(k,a)$-narrow}
in a graph $G$ if there exists a set $D\subseteq V(G)$ of size at most $k$ such that each vertex of $Z$ is at distance at most $a$
from $D$.
\begin{obs}\label{obs-narrow}
Let $k\ge 1$ and $a\ge 0$ be integers and let $Z$ be a $(k,a)$-narrow set in a graph $G$.
Then \begin{enumerate}
	\item[(a)] the $G$-diameter of any connected subgraph of $G[Z]$ is at most $2ak+k -1$, and 
	\item[(b)] if $\varphi$ is any coloring of a subgraph $H$ of $G$ containing $Z$
	and the restriction of $\varphi$ to $H-Z$ has $G$-diameter at most $\ell$, then
	$\varphi$ has $G$-diameter at most $k(2a+2\ell+3)$.
	\end{enumerate}
\end{obs}
\begin{proof}
Let $D$ be a set of at most $k$ vertices of $G$ such that each vertex of $Z$ is at distance at most $a$ from $D$.

Consider vertices $u,v\in Z$ joined by a path $P$ in $G[Z]$.  Let $v_0=u$, and for $i=1,\ldots, k$,
let $v'_i$ be the last vertex of $P$ at distance at most $2a$ from $v_{i-1}$, and let $v_i=v$ if $v'_i=v$
and let $v_i$ be the vertex following $v'_i$ in $P$ otherwise. Note that $d_G(v_i,v_{i-1}) \leq 2a+1$ and  if $v_i \neq v$ then $d_G(v_{i},v_{j}) \geq 2a+1$
for every $0 \leq j < i$. 

By the pigeonhole principle there exist $0 \leq j < i \leq k$ such that $v_i$ and $v_j$
are at distance at most $a$ from the same vertex of $D$. Thus  $d_G(v_i,v_j) \leq 2a$ and so $v_i=v$. We conclude that  $$d_G(u,v) \leq d_G(v_i,v_j) + \sum_{s=0}^{j-1}d_G(v_s,v_{s+1})  \leq 2ak+k -1.$$
This finishes the proof of (a).

For the proof of (b), consider any connected monochromatic subgraph $F$ of $H$ under the coloring $\varphi$.  If $F$ is disjoint
from $Z$, then $F$ has $G$-diameter at most $\ell$.  Hence, suppose that $F$ intersects $Z$. Each component of $F-Z$
has $G$-diameter at most $\ell$ and contains a vertex with a neighbor in $Z$, and thus the distance in $G$ from any vertex of $F$
to $Z$ is at most $\ell+1$.  Consequently, the distance in $G$ from any vertex of $F$ to $D$ is at most $a+\ell+1$, and
thus $F$ is $(k,a+\ell+1)$-narrow.  By (a), $F$ has $G$-diameter less than $(2a+2\ell+3)k$, as desired.
\end{proof}

Let $(T,\beta)$ be a tree decomposition of a graph $G$.  We say that a subgraph $H$ of $G$
is \emph{$k$-narrow} in $(T,\beta)$ if for every $x\in V(T)$, the set $\beta(x)\cap V(H)$ is $(k,1)$-narrow in $G[\beta(x)]$.
Note that $G$ is $k$-narrow in $(T,\beta)$ if and only if $(T,\beta)$ is $k$-dominated.  Lemma~\ref{lem-col} now
follows from the following lemma by setting $H=G$, choosing the $k$-dominated tree decomposition $(T,\beta)$ of $G$ and its node $r$
arbitrarily, and letting $Z=\emptyset$.

\begin{figure}
\begin{center}
\includegraphics{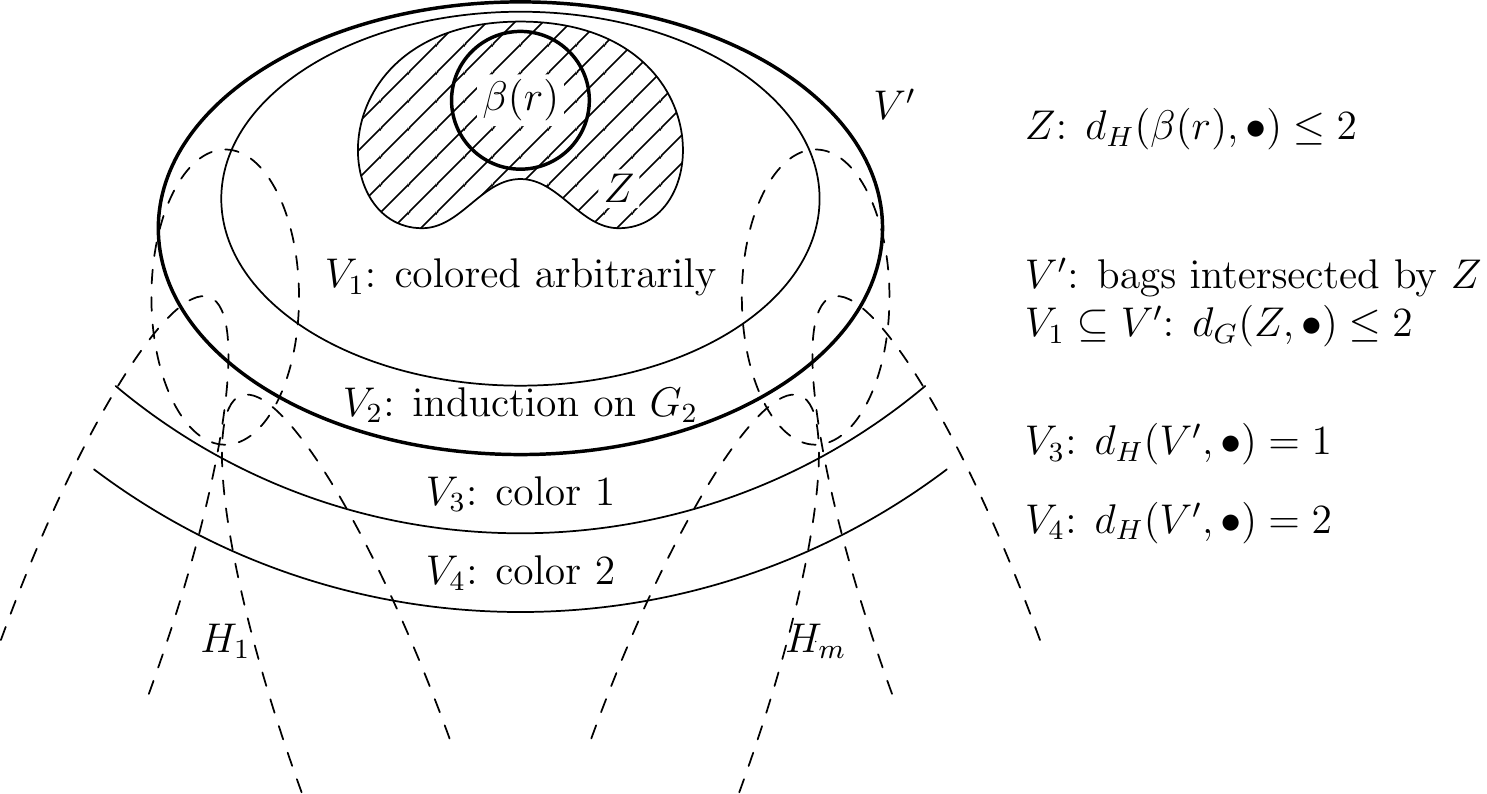}
\end{center}
\caption{Notation used in the proof of Lemma~\ref{lem-col1}}\label{fig-col1}
\end{figure}

\begin{lem}\label{lem-col1}
There exists a function $w: \bb{Z}_+ \to \bb{Z}_+$ satisfying the following for every non-negative integer $k$.  Let $G$ be a graph with tree decomposition $(T,\beta)$, let $H$ be a subgraph of $G$,
let $r$ be a node of $T$, and let $Z\subseteq V(H)$ consist of vertices whose distance in $H$ from $\beta(r)\cap V(H)$
is at most two.  If $H$ is $k$-narrow in $(T,\beta)$, then any $2$-coloring of $Z$ extends to a $2$-coloring of $H$ of $G$-diameter at most $w(k)$.
\end{lem}
\begin{proof}
Let $w(0)=0$ and for $k\ge 1$, let $w(k)=k(2(5k+1)(w(k-1)+2)+13)$.  We prove the claim by induction on $(k,|V(H)|)$
in the lexicographic ordering.  The claim is trivial if $V(H)=\emptyset$;
hence, we can assume $V(H)\neq\emptyset$, and thus $k\ge 1$.
Moreover, we can assume that $\beta(r)\cap V(H)\neq \emptyset$, as otherwise $Z=\emptyset$ and we can
consider the subgraphs of $G$ and $H$ induced by distinct components of $T-r$ separately
(with the node playing the role of $r$ chosen arbitrarily in each component and the empty set playing the role of $Z$);
note that when $\beta(r)\cap V(H)=\emptyset$, each subgraph of $H$ induced by a component of $T-r$
is a disjoint union of components of $H$.

The notation used in the proof is summarized in Figure~\ref{fig-col1}.
We can assume that $Z$ contains all vertices whose distance in $H$ from $\beta(r)\cap V(H)$ is at most two,
as otherwise we can add vertices of $Z$ and give them color $1$.  Let $\psi$ be a given $2$-coloring of $Z$
that we aim to extend to $H$.  Let $D\subseteq \beta(r)$ be a set of size at most $k$ dominating $\beta(r)\cap V(H)$.
Let $S$ consist of all nodes $x\in V(T)$ such that $\beta(x)$ contains a vertex of $Z$;
then $S$ induces a non-empty connected subtree of $T$ rooted in $r$.  Let $V'=\bigcup_{x\in S} (V(H)\cap \beta(x))$
and let $V_1\subseteq V'$ consist of vertices whose distance in $G$ from $Z$ is at most $2$.
Observe that each vertex of $V_1$ is at distance at most $5$ from $D$ in $G$,
and thus $V_1$ is $(k,5)$-narrow in $G$.  Let $\varphi_1$ be any $2$-coloring of $V_1$ extending $\psi$.

Let $V_2=V'\setminus V_1$ and let $G_2$ be the graph obtained from $G\bigl[\bigcup_{x\in S}\beta(x)\bigr]$ by adding the edge $uv$ for each $u,v\in V_2$
such that $u,v\in\beta(x)$ for some $x\in S$ and the distance in $G$ between $u$ and $v$ is at most $5k+1$.
Clearly $(T[S],\beta)$ is a tree decomposition of $G_2$.  Let $H_2=G_2[V_2]$.
We claim that $H_2$ is $(k-1)$-narrow in $(T[S],\beta)$.  Indeed, consider any $x\in S$.
Since $H$ is $k$-narrow in $(T,\beta)$, there exists a set $D_x\subseteq \beta(x)$ of size at most $k$ such that
every vertex of $\beta(x)\cap V(H)$ is at distance at most one from $D_x$.  Since $x\in S$,
there exists a vertex $z\in Z\cap \beta(x)$.  Let $z'\in D_x$ be a vertex at distance at most one from $z$.
All neighbors of $z'$ in $\beta(x)\cap V(H)$ are at distance at most $2$ from $Z$ in $G$ and belong to $V_1$.
Hence, no vertex of $\beta(x)\cap V(H_2)$ is at distance at most one from $z'$, and thus they
are all at distance at most one (in $G$, and thus also in $G_2$) from $D_x\setminus\{z'\}$,
a set of size at most $k-1$.
By the induction hypothesis (applied with the empty set playing the role of $Z$), $H_2$ has a
2-coloring $\varphi_2$ of $G_2$-diameter at most $w(k-1)$.

Let $V_3$ and $V_4$ be the sets of vertices of $H$ at distance (in $H$) exactly one and two from $V'$, respectively.
Note that for any connected component $C$ of $H[V_3]$, there exists $x\in S$ such that each vertex of $C$
has a neighbor in $\beta(x)\cap V(H)$ and no neighbors in $V'\setminus\beta(x)$.  Since $\beta(x)\cap V(H)$
is $(k,1)$-narrow in $G$, $C$ is $(k,2)$-narrow in $G$, and by Observation~\ref{obs-narrow} (a),
$C$ has $G$-diameter at most $5k-1$.  Hence, if distinct vertices $u,v\in V_2$ both have a neighbor in $C$,
then $u,v\in \beta(x)$ and the distance between $u$ and $v$ in $G$ is at most $5k+1$, and thus $uv\in E(G_2)$.
Letting $\varphi_3$ be the coloring that assigns to each vertex of $V_3$ the color $1$,
this implies that for each monochromatic component $C'$ of $\varphi_2\cup\varphi_3$ on $H[V_2\cup V_3]$,
$V(C')\cap V_2$ is a subset of a monochromatic component of $\varphi_2$ on $H_2$.  Consequently, $C'$ has
$G$-diameter at most $(5k+1)(w(k-1)+2)$.  Since $V_1$ is $(k,5)$-narrow in $G$, Observation~\ref{obs-narrow} (b) implies that
the coloring $\varphi'=\varphi_1\cup\varphi_2\cup \varphi_3$ of $H[V'\cup V_3]$ has $G$-diameter at most
$k(2(5k+1)(w(k-1)+2)+13)=w(k)$.

Let $T_1$, \ldots, $T_m$ be the components of $T-S$ and for $i=1,\ldots,m$, let $r_i$ be the unique node of $T_i$ with a neighbor in $S$.
Let $H_i=H\Bigl[\bigcup_{x\in V(T_i)} \beta(x)\cap V(H)\Bigr]$ and $Z_i=(V'\cup V_3\cup V_4)\cap V(H_i)$ and let $\psi_i$ be the coloring
that matches $\varphi'$ on $(V'\cup V_3)\cap V(H_i)$ and that gives all vertices of $V_4\cap V(H_i)$ the color $2$.  Note that
$V'\cap V(H_i)\subseteq \beta(r_i)$.  Clearly, $H_i$ is $k$-narrow in $(T,\beta)$ and $|V(H_i)|<|V(H)|$, and thus
$\psi_i$ extends to a $2$-coloring $\varphi'_i$ of $H_i$ of $G$-diameter at most $w(k)$ by the induction hypothesis.

Finally, define $\varphi=\varphi'\cup\bigcup_{i=1}^m \varphi'_i$.  Note that since all vertices in $V_3$ have color $1$ and all vertices
in $V_4$ have color $2$, each monochromatic component of $\varphi$ on $H$ is a monochromatic component of either $\varphi'$
on $H[V'\cup V_3]$ or of $\varphi'_i$ on $H_i$ for some $i\in\{1,\ldots,m\}$, and thus $\varphi$ has $G$-diameter at most $w(k)$.
\end{proof}

\section{The asymptotic dimension of intersection graphs}\label{sec-wrapup}

We are now ready to prove the main result.

\begin{proof}[Proof of Theorem~\ref{thm-main}]
The space $(X,r)$ has an $n$-dimensional control function $D(r)=Kr$ for some $K>1$; let $C=2K(2K+1)$.
Let $w$ be the function from Lemma~\ref{lem-col}.
For an odd positive integer $r=2t+1$, let $f_r(x)=f((2t+2)(x+1))$ and $D'(r)=w(f_r(C))$.  

Let $\mc{S}$ be an $f$-space-filling system of subsets of $X$
and let $G$ be the intersection graph of $\mc{S}$.  We can assume that all sets in $\mc{S}$ are non-empty,
since an empty set would form an isolated vertex in $G$, not affecting the asymptotic dimension.
We will verify that $G$ satisfies the condition from Lemma~\ref{lemma-asweak}.
It suffices to consider the case that $r$ is odd, since a coloring of $G^{r+1}$ gives also a coloring of $G^r$ whose $G^r$-diameter
is at most twice its $G^{r+1}$-diameter.  By a standard compactness argument~\cite[Theorem A.2]{bonamy2021asymptotic},
we can assume $\mc{S}$ is finite.

Let $\mc{S}^t$ be the system containing for each $S\in \mc{S}$ the union $U_S$ of all sets in $\mc{S}$ at distance at most $t$
from $S$ in $G$, and observe that $G^r$ is exactly the intersection graph of $\mc{S}^t$.
By Lemma~\ref{lem-cons}, the system $\mc{S}^t$ is $f_r$-space-filling.
By Lemma~\ref{lemma-exweb}, $(X,d)$ admits an $(n+1)$-laminar $C$-web $\mc{W}=\mc{W}_1\cup \ldots\cup \mc{W}_{n+1}$,
where $\mc{W}_1$, \ldots, $\mc{W}_{n+1}$ are laminar.  For $i=1,\ldots, n+1$, let $\mc{S}_i$ be the set of elements
$U\in \mc{S}^t$ for which $i$ is the minimum index such that $\mc{W}_i$ $C$-catches $U$,
and let $G_i$ be the intersection graph of $\mc{S}_i$.  Note that $V(G^r)$ is the disjoint union of vertex sets of $G_1$,
\ldots, $G_{n+1}$, and that $G_i$ is an induced subgraph of $G^r$.
By Lemma~\ref{lemma-decomp}, $G_i$ has an $f_r(C)$-dominated tree decomposition, implying by Lemma~\ref{lem-col}
that $G_i$ has a coloring $\varphi_i$ of $G_i$-diameter at most $w(f_r(C))=D'(r)$ using only colors $\{2i-1,2i\}$.
Note that the $G^r$-diameter of $\varphi_i$ is at most as large.

Consequently, $\varphi_1\cup \ldots\cup \varphi_{n+1}$ is a $(2n+2)$-coloring of $G^r$ of $G^r$-diameter at most $D'(r)$.
\end{proof}

\bibliographystyle{siam}
\bibliography{../data.bib}
\end{document}